\documentclass[12pt]{amsart}
\usepackage[english]{babel}
\usepackage{amsmath}
\usepackage{amssymb}

\numberwithin{equation}{section}
\theoremstyle{plain}
\newtheorem{theorem}{Theorem}
\newtheorem{lemma}{Lemma}
\theoremstyle{definition}
\newtheorem{df}{Definition}
\newtheorem{rem}{Remark}
\newcommand{\rbmo}{\mathrm{RBMO}}
\newcommand{\krl}{\mathcal{K}}
\newcommand{\KK}{K}
\newcommand{\dist}{d}
\newcommand{\rd}{\mathbb{R}^m}

\newcommand{\al}{\alpha}

\newcommand{\dou}{\mathcal{D}}

\newcommand{\er}{\varepsilon}

\newcommand{\ph}{\varphi}

\newcommand{\Rbb}{\mathbb R}

\newcommand{\qe}{S}
\begin{document}

\title[A characterization of Calder\'on--Zygmund operators]{A characterization of Calder\'on--Zygmund operators on RBMO}
\author{Evgueni Doubtsov}
\address{St.~Petersburg Department
of Steklov Mathematical Institute, Fontanka 27, St.~Petersburg
191023, Russia}
\email{dubtsov@pdmi.ras.ru}

\author{Andrei V. Vasin}
\address{St.~Petersburg Department
of Steklov Mathematical Institute, Fontanka 27, St.~Petersburg
191023, Russia and
Admiral Makarov State University of Maritime and Inland Shipping,
Dvinskaya st.~5/7, St.~Petersburg 198035, Russia}

\email{andrejvasin@gmail.com}

\thanks{This research was supported by the Russian Science Foundation (grant No.~23-11-00171),
https://rscf.ru/project/23-11-00171/.}


\begin{abstract}
Let $\mathrm{RBMO}(\mu) = \mathrm{RBMO}(\mathbb{R}^m, \mu)$ 
denote the regular BMO space introduced by X.~Tolsa
for an $n$-dimensional finite positive measure on $\mathbb{R}^m$, $0<n \le m$.
We characterize the bounded Calder\'on--Zygmund operators
$T:\rbmo(\mu) \to \rbmo(\mu)$
in terms of the function $T1$.
\end{abstract}

\maketitle

\section{Introduction}
%%\label{Int}

The main results of classical harmonic analysis are related to homogeneous spaces $(\rd, \mu)$;
see, for example, monograph \cite{St93}.
By definition, this means that 
$\mu$ has the standard doubling property $\mu(2Q) \le C \mu(Q)$
for an arbitrary cube $Q\subset \rd$.
The principal motivation of the present work is to consider the Calder\'on--Zygmund operators in a non-classical setting, 
that is, on non-homogeneous spaces.
Indeed, the main object in this work is the functional space $\rbmo(\rd, \mu)$
introduced by X.\,Tolsa \cite{T} for an $n$-dimensional measure $\mu$.
Thus, the measure $\mu$ may not have the standard doubling property.
Other generalizations of the space $\mathrm{BMO}$ were also considered
for such measures;
see, for example, \cite{MMNO2k, NTV03}.
However, $\rbmo(\rd, \mu)$ has various genuine properties of the classical space $\mathrm{BMO}$, 
for example, the John--Nirenberg inequality holds for $\rbmo(\rd, \mu)$.

Let $T$ be a Calder\'on--Zygmund operator
such that its kernel satisfies the appropriate cancellation property
(see \eqref{e_cz2}).
In \cite{DV23PAMS}, the authors obtained a condition on the function $T1$, which 
guarantees that the operator $T$ is bounded on
$\rbmo(\mu) = \rbmo(\rd, \mu)$.
In this work, we prove that the condition from \cite{DV23PAMS},
after a slight modification, is not only sufficient, but also necessary for the boundedness of $T$.
In other words, in the present paper,
we obtain a so-called $T1$ theorem for the operator $T$ on $\rbmo(\mu)$,
that is, we give a characterization of the bounded Calder\'on--Zygmund operators
$T:\rbmo(\mu) \to \rbmo(\mu)$
in terms of the function $T1$.

We introduce certain definitions and facts from \cite{T} and \cite{DV23PAMS}.

\subsection{Cubes and $n$-dimensional measures}
Unless otherwise stated, everywhere below $\mu$ is a finite positive measure defined on $\rd$.
By definition, cube is a closed cube in $\rd$, its sides are parallel to the coordinate axes
and its center belongs to the support of $\mu$. 
The side length of $Q$ is denoted by $\ell=\ell(Q)$.

As in \cite{T}, we always assume that $\mu$ is an $n$-dimensional measure on $\rd$, where
$0< n \le m$.
By definition, it means that
\begin{equation}  \label{e_ndim}
\mu(Q) \le C \ell^n(Q)\ \textrm{for every cube\ } Q\subset\rd,\ \ell(Q)>0,
 \end{equation}
with a constant $C>0$.

\subsection{Calder\'on--Zygmund operators.} \, 
A Calder\'on--Zygmund kernel associated with an $n$-dimensional measure $\mu$ on $\rd$
is a measurable function
$\krl (x, y)$ defined on $\rd\times \rd\setminus\{(x, x): x\in \rd\}$
and satisfying the following conditions:
\begin{equation}\label{e_cz1}
|\krl(x, y)|
\le C \dist^{-n}(x,y),
\end{equation}
\begin{equation}\label{e_cz3}
\aligned
 |\krl(x_1, y)- \krl(x_2, y)|
&+ |\krl(y, x_1) - \krl(y, x_2)| \\
&\le C\frac{\dist^\delta(x_1, x_2)}{\dist^{n+\delta}(x_1, y)}, \quad\textrm{whenever}\ \ 2\dist(x_1, x_2)\le \dist(x_1, y),
\endaligned
\end{equation}
where $C>0$ is a universal constant and
$\delta$, $0 < \delta\le 1$, is a regularity constant.
Restrictions \eqref{e_cz1} and \eqref{e_cz3} are standard by now.
Also, remark that property~\eqref{e_cz3} 
is used in the celebrated $T1$ theorem of David and Journ\'e \cite{DJ84}.

Also, in the main theorem of the present paper, we use the following more specific
cancellation property:
\begin{equation}\label{e_cz2}
\left|\int_{Q(x, R)\setminus Q(x, r)} \krl(x, y)\, d\mu(y) \right|
\le C, \quad 0<r<R,
\end{equation}
where $C>0$ is a universal constant. 
This type of restriction is sufficiently common 
(see, for example, \cite[Ch.~7, Exercise.~5.17]{St93} in the classical setting,
see also \cite{Ga09} in the case of nonhomogeneous spaces).

The Calder\'on--Zygmund operator associated to the kernel $\krl (x, y)$ and the measure
$\mu$ is defined as
\[
Tf(x)=\int_{\rd} \krl (x, y)f(y)\, d\mu(y) 
\]
for $x\notin \textrm{supp}(f\mu)$.
In the general situation, one introduces the truncated operators
 $T_\er$, $\er>0$:
\[
T_\er f(x)=\int_{\rd\setminus Q(x, \er)}\krl (x, y)f(y)\, d\mu(y).
 \]
We say that $T$ is bounded on $L^p(\mu)$
provided that the operators $T_\er$ are bounded on $L^p(\mu)$ uniformly in $\er >0$.

\subsection{Regular BMO space}
To define $\rbmo(\mu)$, we need coefficients $\KK(Q, R)$
and doubling cubes.

\subsubsection{Coefficients $\KK(Q, R)$}
For embedded cubes $Q\subset R$ in $\rd$, put
\[
\KK(Q, R) = 1 + \sum_{j=1}^{N_{Q, R}} \frac{\mu(2^j Q)}{\ell^n(2^j Q)},
\]
where $N_{Q, R}$ is the minimal
integer $k$ such that $\ell(2^k Q) \ge \ell(R)$.
Clearly, $\KK(Q, R)\ge 1$.
Also, remark that $\KK(Q, R)$ is estimated by
$C \log(\ell(R)/ \ell(Q))$, since $\mu$ is $n$-dimensional.

\subsubsection{Doubling cubes}
%\begin{df}
Let $\al>1$, $\beta > \al^n$. A cube $Q$ is called $(\al, \beta)$-doubling if
\[
\mu(\al Q) < \beta \mu(Q).
\]
%\end{df}
Let $\mu$ be a Radon measure defined on $\rd$ and $\al>1$.
As indicated in \cite{T}, it is known that, for a sufficiently large parameter
$\beta = \beta (\al, n)$, for $\mu$-almost all $x\in \rd$
there exists a sequence of $(\alpha, \beta)$-doubling cubes $\{Q_k\}_{k=1}^\infty$ centered at $x$ 
and such that 
$\ell(Q_k) \to 0$ as $k\to \infty$.
For definiteness, let $\beta_n$ denote the double infimum of the corresponding constants
 $\beta(10, n)$.

\begin{df}\label{d_doubling}
A cube $Q\subset\rd$ 
is called \textit{doubling} (we write $Q\in\dou$)
if $Q$ is $(10, \beta_n)$-doubling.
\end{df}

\begin{rem}
The analogous definition in 
\cite{DV23PAMS} is based on 
$(4, \beta_n)$-doubling cubes in the place of $(10, \beta_n)$-doubling ones.
Also, remark that Tolsa \cite{T} uses $(2, \beta)$-doubling cubes in his original definition.
However, further work showed that
in Tolsa theory in $\rd$ one can use $(\alpha, \beta)$-doubling cubes with $\alpha >1$
(see, for example, \cite{HMY13, HYY12}).
\end{rem}

\subsubsection{Definition of the space $\rbmo(\mu)$}
 The definition of the space $\rbmo(\mu)$ given below is used as the main one in \cite{DV23PAMS}.
Remark that Tolsa \cite{T} used a different equivalent definition as his principal one.
 
\begin{df}\label{d_rbmo}
The space $\rbmo(\mu)$ consists of $f\in L^1_{loc}(\mu)$ such that
there exists a constant $C_{\mathfrak{E}}>0$
and a collection of constants $\{f_Q\}$ 
(one constant for every doubling cube $Q\subset \rd$)
such that
\begin{align}
 \frac{1}{\mu(Q)}
 \int_Q |f - f_Q|\, d\mu
&\le C_{\mathfrak{E}}, \label{e_df_osc}
 \\
 |f_Q - f_R|
&\le C_{\mathfrak{E}} \KK(Q, R) \label{e_df_K}
\end{align}
for all cubes $Q,R\in\dou$, $Q\subset R$.
Let $\|f\|_{\mathfrak{E}}$ denote the infimum of the 
corresponding constants $C_{\mathfrak{E}}>0$.
\end{df}

It is not difficult to verify that 
$\|f\|_{\mathfrak{E}}$ is a norm on the space $\rbmo(\mu)$
modulo constants.
Next, $\mu$ is a finite $n$-dimensional measure, hence,
$\rbmo(\mu) \subset L^1(\mu)$ by standard arguments.
Direct inspection shows that $\rbmo(\mu)$
is a Banach space with respect to the following norm:
\[
\|f\|_*:=\|f\|_{\mathfrak{E}} + \|f\|_{L^1(\mu)}.
\]

\subsection{Main theorem}\label{ss_mthm}
Recall that $\mu(\rd)< \infty$ by assumption.
For a cube $Q\subset \rd$, put
\[
K(Q)=K(Q,2^k Q),
\]
where $k$ is the minimal positive integer such that $\mu(2^k Q)>\frac{1}{2}\mu(\rd)$.

\begin{theorem}\label{t_main}
Let $\mu$ be a finite positive $n$-dimensional measure defined on $\rd$, $0< n \le m$.
Let $T$ be a Calder\'on--Zygmund operator such that the cancellation property \eqref{e_cz2} holds.
Then the following properties are equivalent:
\begin{itemize}
  \item [(i)] The operator $T$ is bounded on the Banach space $\rbmo(\mu)$.
  \item [(ii)] For every cube $Q\in\dou$, there exists a constant $b_Q$ such that
\begin{align}
\frac{1}{\mu(Q)}
\int_Q |T1 - b_Q|\, d\mu
&\le \frac{C}{\KK(Q)}\quad\textrm{for all cubes}\ Q\in\dou,\label{e_main_osc} \\
 |b_Q - b_R|
&\le C \frac{\KK(Q, R)}{\KK(Q)}\quad\textrm{for all cubes}\ Q,R \in\dou,\ Q\subset R,\label{e_main_K}
\end{align}
where the constant $C>0$ does not depend on $Q$ and $R$.
\end{itemize}
\end{theorem}

\begin{rem}
By definition, $T$ is bounded on   
$RBMO(\mu)$
provided that the truncated operators $T_\er$, $\er > 0$, 
are uniformly bounded on the space $RBMO(\mu)$. 
Analogously, by definition, 
inequalities \eqref{e_main_osc} and \eqref{e_main_K} 
mean that these estimates hold for the operators $T_\er$ 
uniformly in $\er > 0$.
\end{rem}
\begin{rem}
Firstly, under hypothesis of Theorem~\ref{t_main}, one implicitly assumes that $T1\in L^\infty(\mu)$.
Indeed, this property follows from \eqref{e_cz2} by \cite[Lemma~2.6]{DV23PAMS}.
Secondly, by Lemma~2.8 from \cite{DV23PAMS} condition \eqref{e_cz2}
implies that $T$ is a bounded operator on $L^2(\mu)$.
\end{rem}

\subsection{Organization of the paper}
Certain auxiliary results are collected in Section~\ref{s_aux}.
In Section~\ref{s_test}, we construct test functions. These functions
are used in the final Section~\ref{s_main} to prove the implication
(i)$\Rightarrow$(ii) in the main theorem.

\subsection{Notation} 
The symbol $C$ is used to denote a positive constant whose value may vary from line to line.
If $A, B>0$ and $A\le C B$ for a constant $C>0$,
then we write $A\lesssim B$.
Notation $A\approx B$ means that $A\lesssim B$ and $B\lesssim A$.

\section{Auxiliary results}\label{s_aux}
\subsection{An equivalent norm on $\rbmo(\mu)$}
\begin{df}\label{d_rbmo_A}
Let $f\in L^1_{loc}(\mu)$. Fix a constant $\rho>1$.
Let $\|f\|_{\mathfrak{A}, \rho}$ denote the infimum of the constants
$C_{\mathfrak{A}}= C_{\mathfrak{A}, \rho}>0$
with the following properties:
for every cube $Q$, there exists a constant $f_Q \in \Rbb$ such that
\begin{align}
\sup_Q
\frac{1}{\mu(\rho Q)}
 \int_Q |f(x) - f_Q|\, d\mu (x)
&\le C_{\mathfrak{A}}, \label{e_A_osc}
\\
 |f_Q - f_R|
&\le C_{\mathfrak{A}} K(Q,R)\quad \text{for all cubes\ } Q \subset R. \label{e_A_K}
\end{align}
Direct inspection shows that  
$\|\cdot\|_{\mathfrak{A}}$ is a seminorm on $\rbmo(\mu)$.
\end{df}

\begin{rem}
The original seminorm introduced by Tolsa \cite{T} on $\rbmo(\mu)$ 
differs from the seminorms given in Definitions~\ref{d_rbmo} and \ref{d_rbmo_A}. 
However, it is known that all these seminorms are equivalent (see \cite{T} and \cite[Section~2]{DV23PAMS}).
\end{rem}

\subsection{Decomposition of a function from $\rbmo(\mu)$ into three terms}\label{ss_cstr}
Let $T$ denote a Carder\'on--Zygmund operator bounded on $L^2(\mu)$.
Assume that $f\in \rbmo(\mu)$ and $\rho=2$.
Applying Definition~\ref{d_rbmo_A}, for every cube $2Q\subset\rd$,
we select a constant $f_{2Q}$ 
such that properties \eqref{e_A_osc}
and \eqref{e_A_K}
hold for the constant $C_{\mathfrak{A}, \rho} = 2\|f\|_{\mathfrak{A}, \rho}$
and for $2Q$ in the place of $Q$, that is,
\begin{align}
\sup_{2Q}
\frac{1}{\mu(2\rho Q)}
 \int_{2Q} 
&|f(x) - f_{2Q}|\, d\mu (x)
 \le 2\|f\|_{\mathfrak{A}, \rho}, \label{e_A_osc_2Q}
\\
 |f_{2Q} - f_{2R}|
&\le 2\|f\|_{\mathfrak{A}, \rho} K(2Q,2R)\quad \text{for all cubes\ } Q \subset R. \label{e_A_K_2Q}
\end{align}

Introduce the following functions:
\[
\begin{aligned}
 f_1 = f_{1,Q}
&= f_{2Q}, \\
 f_2 = f_{2,Q}
&= (f- f_{2Q}) \chi_{2Q}, \\
 f_3 = f_{3,Q}
&= (f- f_{2Q}) \chi_{\rd\setminus 2Q}.
\end{aligned}
\]

Observe that
$
f= f_1 + f_2 + f_3.
$
Similar decompositions are well known (see, for example, \cite{H}).
%see also \cite{DV}.

Put $b_{2, Q}=0$ and
\[
b_{3, Q} = \frac{1}{\mu(Q)} \int_{Q} Tf_{3,Q} (y)\, d\mu(y).
\]

We will need two lemmas from \cite{DV23PAMS} about properties of the above functions and constants.

\begin{lemma}\label{l_23}
There exists a constant $C>0$ such that
\[
\frac{1}{\mu(Q)} \int_Q |T f_k  - b_{k, Q}|\,d\mu \le C \|f\|, \quad k= 2,3,
\]
for any cube $Q\in\dou$.
\end{lemma}

\begin{lemma}\label{l_23_K}
There exists a constant $C>0$ such that
\[
|b_{k, Q} - b_{k, R}| \le C \|f\| \KK(Q, R), \quad k= 2, 3.
\]
for any cubes $Q\subset R$.
\end{lemma}

\section{Test functions}\label{s_test}

In this section, we construct test functions used to prove the implication (i)$\Rightarrow$(ii).
We introduce the following notation:
\[
|x|=\max \{|x_j|: j=1,2,\dots, m\}, \quad x\in\rd.
\]
Put
\begin{equation}\label{e_ph_df}
  \varphi(y)=\varphi(|y|)=1+\int_{|y|<|t|} \frac{d\mu(t)}{|t|^n}, \quad y\in\rd \setminus \{0\}.
\end{equation}
The measure $\mu$ is $n$-dimensional and $\mu(\rd)<\infty$, thus, the above integral converges.
The function $\varphi$ is continuous and $\ell^\infty$-radial on its domain.

Let $Q(r)$, $r>0$, denote the
cube centered at the origin and with side length $2r$.

\begin{rem}\label{r_r4}
Using the dyadic decomposition for the integral defining $\ph(y)$, we conclude that
\[
 \varphi(y)\approx K(Q(|y|))
\]
for $y\in \mathbb{R}^d\setminus \{0\}$. See \cite{T}.
\end{rem}

The following result is motivated by Lemma~1.1
from \cite{Sj}, where 
$\mu$ is Lebesgue measure.

\begin{lemma}\label{l_test}
Let $\ph$ be defined by \eqref{e_ph_df}, $\alpha\ge 5$ and $\beta> \alpha^n$.
Then $\varphi\in \rbmo(\mu)$. 
There exist constants $C$ and $C(\alpha, \beta)>0$ such that
\begin{equation}\label{e_ph_Q}
|\langle\ph_\qe \rangle|\ge C K(\qe) - C(\alpha, \beta)
\end{equation}
for every $(\alpha, \beta)$-doubling cube $\qe$ centered at the origin.
\end{lemma}

\begin{rem}
We apply Lemma~\ref{l_test} with a sufficiently large $K(\qe)$, therefore,
 $C K(\qe) - C(\alpha, \beta)$
in estimate~\eqref{e_ph_Q}
becomes positive in this case.
\end{rem}

\begin{proof}[Proof of Lemma~\ref{l_test}]
The argument is based on the following plan: on the first two steps, for an arbitrary cube $Q$, 
we verify conditions \eqref{e_A_osc} and \eqref{e_A_K} from Definition~\ref{d_rbmo_A} with $\rho=5$.
The final third step is devoted to proving estimate \eqref{e_ph_Q} 
for an $(\alpha, \beta)$-doubling cube $\qe$ centered at the origin.

\smallskip
\textbf{Step I.}
On this step, we verify the oscillation condition \eqref{e_A_osc} from Definition~\ref{d_rbmo_A} with $\rho=5$. 
To do this, fix a cube $Q$ with center $x_Q$ and with side length $\ell_Q$.
Let $\widetilde {Q}$ denote the smallest cube with center at the origin and containing the cube $Q$. 
Choose $y_Q\in \partial Q$ such that $|y_Q|=\ell_Q/2+|x_Q|$. Put $\ph_Q=\varphi(y_Q)$. We have
\[
\begin{split}
I
&=\int_Q|\varphi(z)-\varphi(y_Q)|d\mu(z) =\int_Q\left|\int_{|t|>|z|} \frac{d\mu(t)}{|t|^n} -\int_{|t|>|y_Q|} \frac{d\mu(t)}{|t|^n}\right|d\mu(z) \\
&=\int_Q d\mu(z) \int_{|z|<|t|<|y_Q|} \frac{d\mu(t)}{|t|^n} 
\\
&\le \int_{\widetilde{Q}} d\mu(z) \int_{|z|<|t|<|y_Q|} \frac{d\mu(t)}{|t|^n} .
\end{split}
\]

Now, consider two cases depending on the fulfillment of the inequality 
\[
\ell_Q> |x_Q|.
\]

If $\ell_Q> |x_Q|$, then
the inequalities $|y_Q|=|x_Q|+\ell_Q/2< 3 \ell_Q/2$ and $|y_Q|+|x_Q|<5\ell_Q/2$ 
imply that $\widetilde{Q}=Q(|y_Q|)$.
Therefore, 
$Q\subset\widetilde{Q}\subset 5Q$. Hence,
 \[
 I\leq \int_{\widetilde {Q}} d\mu(z) \int_{|z|<|t|<|y_Q|} \frac{d\mu(t)}{|t|^n} 
 = \int_{|z|<|y_Q|} d\mu(z) \int_{|z|<|t|<|y_Q|} \frac{d\mu(t)}{|t|^n}. 
 \]
Changing the order of integration, we continue as
\[
I\le \int_ {|t|<|y_Q|}\frac{d\mu(t)}{|t|^n} \int_{|z|<|t|}  d\mu(z)
\le\int_ {|t|<|y_Q|}\frac{\mu(Q(|t|))}{|t|^n}d\mu(t).  
\]
Since $\mu$ is $n$-dimensional, we have 
 $\mu(Q(|t|))\leq C |t|^n$, therefore,
\[I\lesssim\int_ {|t|>|y_Q|}d\mu(t) =\mu(\widetilde{Q})\leq \mu(5Q),
\]
as required. The analysis of the first case is completed.

If $\ell_Q\leq |x_Q|$, then $|z|\geq |x_Q|/2$ for $z\in Q$. Also, we have $|y_Q|\leq 3|x_Q|/2$.
Hence, for the inner integral, we obtain
\[
\begin{split}
\int_{|z|<|t|<|y_Q|} \frac{d\mu(t)}{|t|^n}
&\leq\int_{|x_Q|/2<|t|<3|x_Q|/2} \frac{d\mu(t)}{|t|^n}
 \\
&\leq \frac{2^n}{|x_Q|^n}\int_{|t|<3|x_Q|/2} d\mu(t) 
\\
&= \frac{2^n}{|x_Q|^n}\mu(Q(3|x_Q|/2)
\leq \frac{3^n}{|\ell_Q|^n}|x_Q|^n \leq C,
\end{split}
\]
since $\mu$ is $n$-dimensional.

Thus, for the entire integral $I$, we have
\[
I\lesssim \int_Q d\mu(z)=\mu(Q),
\]
which completes the analysis of the second case.
Therefore, the oscillation condition from Definition~\ref{d_rbmo_A} 
is verified with parameter $\rho=5$.

\smallskip
\textbf{Step II.} 
On this step, we verify $K$-condition from
Definition~\ref{d_rbmo_A} with $\rho=5$.
Recall that $\ph_Q=\varphi(y_Q)$ for $y_Q\in \partial Q$ such that $|y_Q|=|x_Q|+\ell_Q/2$.

Consider cubes $Q\subset R$. Let $k$ be the minimal integer such that
$Q_0:=2^kQ \supset R$.  
Let $\widetilde{Q}$, 
$\widetilde{Q}_0$ and $\widetilde{R}$ denote the smallest cubes centered at zero
and containing $Q$, 
$Q_0$ and $R$, respectively.

We have
\[
|\ph_Q-\ph_R|
\le |\ph_Q-\ph_{Q_0}|+|\ph_{Q_0}-\ph_R|=A+B.
\]

The cubes $Q_0$ and $Q$ are concentric, thus, $|y_{Q_0}|=|x_Q|+2^{k-1}\ell_Q$. 
Hence, 
\[
A=\left|\int_{|t|>|y_Q|} \frac{d\mu(t)}{|t|} -\int_{|t|>|y_{Q_0}|} \frac{d\mu(t)}{|t|^n} \right|
=\int_{|y_Q|<|t|<|y_{Q_0}|} \frac{d\mu(t)}{|t|^n} \approx K(\widetilde {Q}, \widetilde {Q}_0)
\]
by the dyadic decomposition of the corresponding integral  (cf.\ \cite{T}).

Now, consider two cases depending on the fulfillment of the condition
\[
2|y_Q|\geq |y_{Q_0}|.
\]

If $2|y_Q|\geq |y_{Q_0}|$, then the cubes $\widetilde {Q}$ and $\widetilde {Q}_0$ 
are comparable. Therefore,
\[
K(\widetilde {Q}, \widetilde {Q}_0)\leq C.
\]

If  $2|y_Q|< |y_{Q_0}|$,
then  $|x_Q|<\ell_{Q_0}/2$. Hence,
$0\in Q_0$. Therefore, $Q_0\subset \widetilde {Q}_0\subset 3Q_0$ and we have
\[
\begin{split}
K(\widetilde {Q}, \widetilde {Q}_0)
&\le  K(\widetilde {Q},3Q_0) \\
&\le K(Q,3Q_0) = K(Q,Q_0)+K(Q_0,3Q_0)\le K(Q,Q_0)+C\approx K(Q,R),
\end{split}
\]
since the cubes $Q_0$ and $R$ are comparable.

Next, we have
\[B=
|\ph_{Q_0}-\ph_R|
=\left|\int_{|t|>|y_{Q_0}|} \frac{d\mu(t)}{|t|^n} -\int_{|t|>|y_R|} \frac{d\mu(t)}{|t|^n} \right|
\approx K(\widetilde {R},\widetilde {Q}_0).\]

Consider two cases depending on the fulfillment of the condition
\[
2|y_R|\geq |y_{Q_0}|.
\]

If $2|y_R|\geq |y_{Q_0}|$, then the cubes $\widetilde {R}$ and $\widetilde {Q}_0$ are comparable. 
Hence,
\[K(\widetilde {R},\widetilde {Q}_0)\leq C.\]

Suppose that $2|y_R|< |y_{Q_0}|$.
Since $R\supset Q$ and the cubes $Q$ and $Q_0$ are concentric, we have
\[|x_Q|<|y_R|<|y_{Q_0}|/2=1/2(|x_Q|+\ell_{Q_0}/2).\]
Therefore, $|x_Q|< \ell_{Q_0}/2$ and  $0\in Q_0$. 
Thus, $Q_0\subset \widetilde {Q}_0\subset 3Q_0$ and we obtain
\[
B\leq K(\widetilde {R}, \widetilde {Q}_0)\leq  K(R,3Q_0)
= K(R,Q_0)+K(Q_0,3Q_0)\leq C,
\]
since the cubes $Q_0$ and $R$ are comparable.

\smallskip
\textbf{Step III.}
On this step, we assume that
$\qe$ is a centered at the origin $(\alpha, \beta)$-doubling cube, $\alpha\ge 5$, that is,
$\mu(5\qe) \le \beta \mu(\qe)$. We have
\[
|\langle\ph\rangle_{\qe}| - |\ph_{\qe}| 
\le \frac{1}{\mu(\qe)}\int_{\qe} |\ph - \ph_{\qe}|\, d\mu 
\le \frac{\mu(5\qe)}{\mu(5\qe)\mu(\qe)}\int_{\qe} |\ph - \ph_{\qe}|\, d\mu 
\le \beta C
\]
by the oscillation property proven on Step~I.

Recall that $\ph_\qe=\varphi (y_\qe)$. 
The cube $\qe$ is centered at the origin, thus, $y_\qe = \ell_\qe /2$.
Hence, by Remark~\ref{r_r4}, we obtain
\[
\varphi(y_\qe)=\int_{|y_\qe|<|t|} \frac{d\mu(t)}{|t|^n} \approx K(\qe).
\]
Therefore, 
\[
|\ph_{\qe}| \ge |\langle\ph\rangle_{\qe}| - C\beta \ge C\KK(\qe) - C\beta,
\]
which completes the proof of the lemma.
\end{proof}

\begin{lemma}\label{l_test22}
For a sufficiently small constant $c>0$, let $\varphi_x = c\varphi(\cdot-x)$, $x\in\rd$,
$\alpha\ge 5$ and $\beta> \alpha^n$.
Then $\|\varphi_x\|_* \le 1$. 
There exist constants $C$ and $C(\alpha, \beta)>0$ such that 
\[
|\langle\ph_x\rangle_Q |\ge C K(Q) - C(\alpha, \beta)
\]
uniformly with respect to $x \in\rd$ 
for every $(\alpha, \beta)$-doubling cube $Q$ centered at $x$.
 \end{lemma}
\begin{proof}
Repeating the arguments from the proof of Lemma~\ref{l_test} for an arbitrary point $x\in\rd$, 
we see that the constants at all steps can be chosen independently of $x$.
This observation provides the required estimate for $|\langle\ph_x\rangle_Q |$ as well as
a bound for the oscillation and $K$-condition. 
To estimate $\|\varphi_x\|_*$, it remains to obtain a uniform bound for $\|\varphi_x\|_{L^1(\mu)}$.

Applying Fubini's theorem, we obtain
\begin{equation}\label{e_L1}
\begin{split}
\left\| \ph_x(y) -1 \right\|_{L^1(\mu)}=
   \int_{\rd} \int_{|y-x|<|t|} \frac{d\mu(t)}{|t|^n} d\mu(y) 
&= 
   \int_{\rd} \int_{|y-x|<|t|} \frac{d\mu(y)}{|t|^n} d\mu(t)
 \\
&\le \int_{\rd} C d\mu(t) \le C,
\end{split}
\end{equation}
since $\mu$ is an $n$-dimensional measure and $\mu(\rd)=1$.
Observe that the above argument does not depend on $x\in\rd$.
Thus, multiplying $\ph$ by a sufficiently small constant,
we guarantee that $\|\ph_x\|_* \le 1$.
The proof of the lemma is complete.
\end{proof}

\begin{lemma}\label{c_test}
There exists a family $\{\varphi_{x}\}_{x\in\rd}$ of test functions such that
$\|\varphi_x\|_* \le 1$ and
\[
|\langle\ph_x\rangle_{2Q}| \ge C \KK(2Q) - C
\]
for every cube $Q\in\dou$ centered at $x$.
% with a universal constant $C>0$.
\end{lemma}
\begin{proof}
By definition, property
$Q\in\dou$ means that $Q$ is a $(10,\beta_n)$-doubling cube.
Therefore, applying Lemma~\ref{l_test22} to $2Q$, a $(5, \beta_n)$-doubling cube,
completes the proof.
Indeed, the constant $\beta_n$ depends on $\mu$ only.
\end{proof}

\section{Proof of Theorem~\ref{t_main}}\label{s_main}

{(ii)$\Rightarrow$(i)} 
This implication is proven in \cite{DV23PAMS} up to the following details:
the arguments in \cite{DV23PAMS} are given for other doubling cubes,
namely, 
by definition, these are $(4, \beta_n)$-doubling cubes.
In fact, the proof remains the same for $(\alpha, \beta_n)$-doubling cubes, $\alpha\ge 4$. 
Thus, to prove the required implication, it is enough to repeat the argument from \cite{DV23PAMS} for $\alpha=10$, 
using the definition of the doubling cubes from the present work.

{(i)$\Rightarrow$(ii)} 
For two doubling cubes $Q\subset R\subset\rd$, 
 we have to verify that properties~\eqref{e_main_osc} and \eqref{e_main_K} hold with appropriate 
 constants $b_Q, b_R$.
Without loss of generality, we assume that $\KK(Q)$ is sufficiently large.

Let $x$ denote the center of $Q$ and let 
$\ph=\ph_x$ denote the test function provided by Lemma~\ref{c_test}.
By property~(i), we have $\|T\ph\|_* \le C \|\ph\|_*$, in particular,
\[
\|T\ph\|_{\mathfrak{}} \le C \|\ph\|_* \le C
\]
Therefore, there exist constants  
$\beta_{Q}$ and $\beta_R$ such that
\begin{align}
 \frac{1}{\mu(Q)}\int_Q |T \ph - \beta_{Q}|\, d\mu
&\le C,\label{e_oscTph}\\
 |\beta_{Q} - \beta_{R}|
&\le C\KK(Q, R), \label{e_K_Tph}
\end{align}
where the constant $C>0$ does not depend on $Q$ and $R$.

Applying the construction from Section~\ref{ss_cstr}, we obtain a decomposition
\begin{equation}\label{e_ph_decomp}
\ph = \ph_{2Q} + \ph_{2} + \ph_{3}
\end{equation}
and constants $b_{2, Q}$ and $b_{3, Q}$.
Observe that, generally speaking, $\ph_{2Q}$ differs from the analogous constant constructed for the cube $2Q$
in the proof of Lemma~\ref{l_test}.
By \eqref{e_A_osc_2Q} with $\ph$ in the place of $f$, we have
\[
\frac{1}{\mu(4Q)} \int_{2Q} |\ph - \ph_{2Q}|\, d\mu \le 2\|\ph\|_{\mathfrak{A}, 2}.
\]
The cube $Q$ is doubling, thus, $\mu(4Q) \le \mu(10 Q) \le \beta_n \mu(Q)\le \beta_n \mu(2Q)$, hence,
\[
|\langle\ph\rangle_{2Q}| - |\ph_{2Q}| \le \frac{\mu(4Q)}{\mu(2Q)\mu(4Q)}\int_{2Q} |\ph - \ph_{2Q}|\, d\mu 
\le 2\beta_n \|\ph\|_{\mathfrak{A}, 2}.
\]
Therefore, by Lemma~\ref{c_test}, we have
\begin{equation}\label{e_ph2Q}
|\ph_{2Q}| \ge |\langle\ph\rangle_{2Q}| - C\|\ph\| \ge C\KK(2Q) - C - C \ge C\KK(2Q),
\end{equation}
since $\KK(2Q)$ is assumed to be sufficiently large.

Finally, we put 
\[
\gamma_Q= \beta_{Q} - b_{2, Q} - b_{3, Q}
\]
and we directly check the conditions
\eqref{e_main_osc} and \eqref{e_main_K}.

\subsubsection*{Oscillation condition \eqref{e_main_osc}}
Put
$
b_Q = {\gamma_{Q}}/{\ph_{2Q}}.
$
Then we have
\[
\begin{split}
\frac{\ph_{2Q}}{\mu(Q)}\int_Q \left| T 1 - b_Q \right|\, d\mu  
&=
\frac{1}{\mu(Q)}\int_Q |T\ph - \beta_Q  + b_{2, Q} - T\ph_2 + b_{3, Q} - T\ph_3|\, d\mu
\\
&\le C 
%\|\ph\|_*
\end{split}
\]
by decomposition \eqref{e_ph_decomp}, property \eqref{e_oscTph} and Lemma~\ref{l_23}. Thus,
\begin{equation}\label{e_oscT1}
\frac{1}{\mu(Q)}\int_Q \left| T 1 - b_Q \right| \, d\mu\le
\frac{C}{|\ph_{2Q}|} \le \frac{C}{\KK(Q)}
\end{equation}
by \eqref{e_ph2Q}.
Therefore, property \eqref{e_main_osc} holds, as required.

\subsubsection*{$\KK$-condition \eqref{e_main_K}}
Applying the definitions of $b_Q$ and $b_R$, we have
\[
\begin{aligned}
 |b_Q - b_R|
\le \left|\frac{\gamma_{Q}}{\ph_{2Q}} - \frac{\gamma_{R}}{\ph_{2Q}}\right| +
 \left|\frac{\gamma_{R}}{\ph_{2Q}}- \frac{\gamma_{R}}{\ph_{2R}}\right| 
&= \frac{|\gamma_{Q} - \gamma_{R}|}{|\ph_{2Q}|} +
 |b_{R}|\frac{|\ph_{2R} - \ph_{2Q}|}{|\ph_{2Q}|} \\
&:= A +B.
 \end{aligned}
\]
Combining Lemmas~\ref{l_23_K} and \ref{l_test} and properties~\eqref{e_K_Tph} and \eqref{e_ph2Q}, we obtain
\[
A \le C\frac{\KK(Q,R)}{\KK(2Q)}.
\]
Since $T 1 \in L^\infty (\mu)$, property \eqref{e_oscT1} implies the estimate $|b_R|\le C$.
Hence, applying \eqref{e_A_K_2Q} and the estimate
 $\|\ph\|_{\mathfrak{A}, 2}\le C$, we obtain
\[
B \le 
 C\frac{\KK(2Q, 2R)}{\KK(2Q)}.
\]
In sum, we have
\[
|b_Q - b_R| \le A +B \le C \frac{\KK(Q,R)}{\KK(Q)},
\]
that is, condition~\eqref{e_main_K} holds.
The proof of the theorem is finished.

\bibliographystyle{amsplain}

\end{document}